\documentclass[12pt, twoside, leqno]{article}
\usepackage{amsmath,amsthm}
\usepackage{amssymb,latexsym}
\usepackage{enumerate}
\usepackage{hyperref,amsbsy}

\newtheorem{thm}{Theorem}[section]
\newtheorem{cor}[thm]{Corollary}
\newtheorem{lem}[thm]{Lemma}
\newtheorem{hip}[thm]{Conjecture}

\theoremstyle{definition}
\newtheorem{defin}[thm]{Definition}
\newtheorem{rem}[thm]{Remark}

\begin{document}

\title{Note on the Davenport's constant for finite abelian groups with rank three}

\author{Maciej Zakarczemny (Cracow)}

\date{}

\maketitle

\renewcommand{\thefootnote}{}



\renewcommand{\thefootnote}{\arabic{footnote}}
\setcounter{footnote}{0}

	\begin{abstract}
	\noindent Let $G$ be a finite abelian group and $D(G)$ denote the Davenport constant of $G$.
	We derive new upper bound for the Davenport constant for all groups of rank three.
	Our main result is that:
	$$D(C_{n_1}\oplus C_{n_2}\oplus C_{n_3})\le (n_1-1)+(n_2-1)+(n_3-1)+1+ (a_3-3)(n_1-1),$$
where $1<n_1|n_2|n_3\in\mathbb{N}$ and $a_3\le 20369$ is a constant.\\Therefore $D(C_{n_1}\oplus C_{n_2}\oplus C_{n_3})$ grows linearly with the variables $n_1,n_2,n_3.$ The new result is the given upper bound for $a_3$.\\
Finally, we give an application of the Davenport constant to smooth numbers.
	\end{abstract}

\section{Introduction}
${}^{}$\indent We study the Davenport constant, a central combinatorial invariant
which has been investigated since Davenport popularized it in the 60's, see \cite{GGZ}, \cite{GHK}, \cite{RG} for a survey.
We derive new explicit upper bound for the Davenport constant for groups of rank three.
The exact value of the Davenport constant for groups of rank three is still unknown and this is an open and well-studied problem, see \cite{BGAA},\cite{BGWS1},\cite{BGWS2}.

\section{Basic notations}
${}^{}$\indent Let $\mathbb{N}$ denote the set of the positive integers (natural numbers). We set \\
$[a,b]=\{x:a\le x\le b,\,\,x\in\mathbb{Z}\},$ where $a,b\in\mathbb{Z}.$
Let $G$ be a non-trivial additive finite abelian group. $G$ can be uniquely decomposed as a direct sum of cyclic groups
$C_{n_1}\oplus C_{n_2}\oplus \ldots \oplus C_{n_r}$ with the integers satisfying $1<n_1|\ldots|n_r.$ The number of summands in the above decomposition of $G$ is denoted by $r = r(G)$ and called the rank of $G$.
The integer $n_r$ denotes the exponent $\exp(G).$ In addition, we define $D^*(G)$ as $D^*(G) =1+ \sum\limits_{i=1}^r (n_i- 1)$.\\
We denote by $\mathcal{F}(G)$ the free, abelian, multiplicatively written monoid with basis $G.$ An element $S\in \mathcal{F}(G)$ is called sequence over $G$. We write any finite sequence $S$ of $l$ elements of $G$ in the form $\prod_{g\in G}g^{\nu_g(S)}=g_1\cdot\ldots\cdot g_l,$
where $l$ is the length of $S,$ denoted by $|S|; \nu_g(S)$ is the multiplicity of $g$ in $S.$
The sum of $S$ is defined as $\sigma (S)=\sum_{g\in G}\nu_g(S)g.$\\
Our notation and terminology is consistent with \cite{RG} and \cite{EEGKR}.\\
The Davenport constant $D(G)$ is defined as the smallest natural number $t$ such that each sequence over $G$ of length at least $t$
has a non-empty zero-sum subsequence. Equivalently, $D(G)$ is the maximal length of a zero-sum sequence of the elements of $G$ and with no proper
zero-sum subsequence. The best bounds for $D(G)$ known so far are:
\begin{equation}\label{xeq1}
D^*(G)\le D(G)\le \exp(G)\left(1+\log{\tfrac{|G|}{\exp(G)}}\right).
\end{equation}
See \cite[Theorem 7.1]{EBKAIII} and \cite[Theorem 1.1]{AGP}.

\section{Theorems and definitions}
\begin{defin}\label{th8}
For an additive finite abelian group $G$, and $m\in\mathbb{N}.$\\ We denote by:
\begin{itemize}
\item[1.] $D_m(G)$ the smallest natural number $t$ such that every sequence $S$ over $G$ of length $|S|\ge t$ contains at least $m$ disjoint and non-empty subsequences $S_1',S_2',\ldots,S_m'|S$ such that
$\sigma(S_i')=0$ for $i\in [1,m].$
\item[2.] $\eta (G)$ the smallest natural number $t$ such that every sequence $S$ over $G$ of length $|S|\ge t$ contains a non-empty subsequences $S'|S$ such that $\sigma(S')=0$, $|S'|\in [1,\exp(G)].$
\item[3.] $s(G)$ the smallest natural number $t$ such that every sequence $S$ over $G$ of length $|S|\ge t$ contains a non-empty subsequences $S'|S$ such that $\sigma(S')=0$, $|S'|=\exp(G).$
\end{itemize}
\end{defin}
\begin{rem}\label{th4}
$D_m(G)$  is called the $m$-th Davenport constant and $s(G)$ the Erd\"os-Ginzburg-Ziv  constant.
In this notation $D(G)=D_1(G).$ See also \cite{FH}, \cite{YGQZ}, \cite{FS10}.
\end{rem}
\begin{lem}\label{tomerge}
Let $G$ be a finite abelian group, $H$ subgroup of $G$, and $k$~a~natural number. Then
\begin{equation}\label{xeq2}
D(G)\le D_{D(H)}(G/H),
\end{equation}
\begin{equation}\label{xeq3}
D_{k}(G)\le \exp(G)(k-1)+\eta(G).
\end{equation}
\end{lem}
\begin{proof}
See \cite[Remark 3.3.3, Theorem 3.6]{FS10} and  \cite[Lemma 6.1.3]{GHK}.
\end{proof}
\begin{lem}\label{lem01}
Let $G$ be a finite abelian group.
\begin{enumerate}
\item[1.] If $G=C_{n_1}\oplus C_{n_2}$ with $1\le n_1|n_2,$ then
$$s(G)=2n_1+2n_2-3,\, \eta(G)=2n_1+n_2-2,\, D(G)=n_1+n_2-2=D^*(G).$$
\item[2.] $D(G)\le \eta(G)\le s(G)-exp(G)+1.$
\end{enumerate}
\end{lem}
\begin{proof}
See \cite[Theorem 5.8.3, Lemma 5.7.2]{GHK} .
\end{proof}

\begin{rem}\label{AlDu}
Alon and Dubiner proved that for every natural $r$ and every prime $p$ we have
\begin{equation}\label{AD}
 s(C_p^r)\le c(r) p,
\end{equation}
where $c(r)$ is recursively defined as follows
\begin{equation}\label{AD1}
c(r)=256 r(\log_2 r+5)c(r-1)+(r+1) \,\,\,\mathrm{for}\,\,\, r\ge 2,\,\,c(1)=2,
\end{equation}
There is a misprint in the corresponding formulas \cite[(6)]{AlD}, \cite[(1.4)]{ChMG}.\\
It should be $c(r)=256 r(\log_2 r+5)c(r-1)+(r+1)$ instead of\\
$c(r)=256 (r\log_2 r+5)c(r-1)+(r+1),$ for more details, see \cite[Remark 3.7]{EEGKR}.
Note that $s(C_p^2)=4p-3\le 4p$ (see, Lemma \ref{lem01}), thus we can start a recurrence  with initial term $c(2)=4$
and get $c(3)<20233.005.$
\end{rem}
\begin{rem}\label{AlDu2}
The method use in \cite{AlD} yields that for every natural number $r\ge 1,$ there exists $a_r>0$ such that for every natural number $n$ we have
\begin{equation}\label{xeq5g}
\eta (C_n^r)\le a_r(n-1)+1.
\end{equation}
We identify $a_r$ with its smallest possible value. It is known that
\begin{equation}\label{xeq5c}
2^r-1\le a_r \le (cr\log{r})^r,
\end{equation}
where $c>0$ is an absolute constant. We know also that $a_1=1,\,a_2=3.$\\
See, \cite{BGAA} and Lemma \ref{lem01}.
\end{rem}
\begin{thm}\label{thE14} \textnormal{(Edel, Elscholtz, Geroldinger, Kubertin, Rackam \cite[Theorem 1.4]{EEGKR})}
Let $G=C_{n_1}\oplus\ldots\oplus C_{n_r}$ with $r=r(G)$ and $1<n_1|\ldots|n_r.$\\
Let $b_1,\ldots, b_r\in\mathbb{N}$ such that for all primes $p$ with $p|n_r$ and all $i\in[1,r],$\\
we have $s(C_p^i)\le b_i(p-1)+1.$ Then
\begin{equation}\label{sssG}
s(G)\le \sum\limits_{i=1}^r(b_{r+1-i}-b_{r-i})n_i-b_r+1,
\end{equation}
where $b_0=0.$ In particular, if $n_1=\ldots=n_r=n,$ then $s(G)\le b_r(n-1)+1.$
\end{thm}

\begin{lem}\label{th13a}
Let $n\ge 2$ be a natural number. Then
\begin{equation}\label{xeq5b}
\eta (C_n^3)\le 20369(n-1)+1 \mathrm{\,\,\,\,and\,\,\,\,} s(C_n^3)\le 20370(n-1)+1.
\end{equation}
Therefore $a_3\le 20369.$
\end{lem}
\begin{proof}
For every finite abelian group $G$, by \cite[Theorem 5.7.4]{GHK} we have
\begin{equation}\label{GaoYang}
s(G)\le |G|+\exp(G)-1.
\end{equation}
Thus, if $p$ is a prime number such that $2\le p\le p_{34}=139$ then

\begin{equation}\label{fff1}
s(C_p^3)\le p^3+p-1< 20370(p-1)+1.
\end{equation}
Assume now that $p$ is a prime number such that $p\ge p_{35}=149.$\\
By Remark \ref{AlDu} we have
$s(C_p^3)< 20233.005 \,p<20370(p-1)+1,$
since $p\ge 149.$
Therefore for all primes $p$ we have
$s(C_p^3)< 20370(p-1)+1.$\\
By Lemma \ref{lem01} we also have $s(C_p)=2(p-1)+1,\,s(C_p^2)=4(p-1)+1$\\
for all primes $p.$\\
Hence by Theorem \ref{thE14} we obtain the upper bound
$$s(C_n^3)\le 20370(n-1)+1,$$
for all natural $n\ge 2.$ Thus, by Lemma \ref{lem01} we obtain
$$\eta(C_n^3)\le 20369(n-1)+1,$$
for all natural $n\ge 2.$
\end{proof}
\begin{thm}\label{th14a}
For an abelian group $C_{n_1}\oplus C_{n_2}\oplus C_{n_3}$ such that\\
$1<n_1|n_2|n_3\in\mathbb{N},$ there exists an absolute constant $a_3\le 20369$ such that:
\begin{equation}\label{xeq5}
D(C_{n_1}\oplus C_{n_2}\oplus C_{n_3})\le (n_1-1)+(n_2-1)+(n_3-1)+1+ (a_3-3)(n_1-1).
\end{equation}
\end{thm}
\begin{proof}
This proof is build on well-know strategy. Let $G$ be a non-trivial finite abelian group $C_{n_1}\oplus C_{n_2}\oplus C_{n_3}$ such that $1<n_1|n_2|n_3\in\mathbb{N}$.
We have that the exponent $\exp(G)=n_3.$ Denoting by $H$ a~subgroup of $G$ such that
\begin{equation}\label{xeq6}
H\cong C_{\tfrac{n_2}{n_1}}\oplus C_{\tfrac{n_3}{n_1}},
\end{equation}
where $\tfrac{n_2}{n_1},\tfrac{n_3}{n_1}\in\mathbb{N}.$
The quotient group $G/H\cong C_{n_1}^3.$\\
By Lemma \ref{tomerge} we get
\begin{equation}\label{xeq7}
D(G)\le D_{D(H)}(G/H)=D_{\tfrac{n_2}{n_1}+\tfrac{n_3}{n_1}-1}(C_{n_1}^3),
\end{equation}
since $D(H)=\tfrac{n_2}{n_1}+\tfrac{n_3}{n_1}-1$ (see Lemma \ref{lem01}).\\
By Lemma \ref{tomerge} and (\ref{xeq5g})
\begin{equation}\label{xeq8}
\begin{split}
D(G)&\le \exp(C_{n_1}^3)(\tfrac{n_2}{n_1}+\tfrac{n_3}{n_1}-2)+\eta(C_{n_1}^3)\le \\
&\le n_1(\tfrac{n_2}{n_1}+\tfrac{n_3}{n_1}-2)+a_3(n_1-1)+1=\\
&= (n_1-1)+(n_2-1)+(n_3-1)+1+ (a_3-3)(n_1-1),
\end{split}
\end{equation}
where $a_3$ is a constant. By (\ref{xeq5c}) and (\ref{xeq5b}) we obtain $a_3\le 20369.$
\end{proof}
\begin{rem}\label{th16}
Let $1<n_1|n_2|n_3\in\mathbb{N}.$ By Theorem \ref{th14a} we have
\begin{equation}\label{xeq9b}
D(C_{n_1}\oplus C_{n_2}\oplus C_{n_3})\le 20367(n_1-1)+(n_2-1)+(n_3-1)+1.
\end{equation}
If $n_3>\tfrac{20367(n_1-1)+n_2-1}{\log{n_1}+\log{n_2}}$ then the upper bound in (\ref{xeq9b}) is smaller than the upper bound from (\ref{xeq1}).
See also \cite{ChMG}.
\end{rem}
\begin{cor}\label{nowyap}
Let $n\ge 2$ be a natural number, and let $\omega(n)$ denote the number of distinct prime factors of $n.$ Then
\begin{equation}
3(n-1)+1\le D(C_n^3)\le \min\{20369,3^{\omega(n)}\}(n-1)+1.
\end{equation}
\end{cor}
\begin{proof}
Taking into account the inequality (\ref{xeq1}) and using Theorem \ref{th14a}, we obtain $$3(n-1)+1\le D(C_n^3)\le 20369(n-1)+1.$$
By \cite[Theorem 1.2]{ChMG}, we get
$$D(C_n^3)\le 3^{\omega(n)}(n-1)+1.$$
\end{proof}
Under the assumption that the conjecture of Gao and Thangadurai \cite[Conjecture 0]{GT} is valid, we can surmise that $a_3=8.$
Thus, it seemed desirable to attempt to put the conjecture:
\begin{hip}\label{th15}
Let $G$ be an abelian group $C_{n_1}\oplus C_{n_2}\oplus C_{n_3}$ such that $1<n_1|n_2|n_3\in\mathbb{N}.$
Then
\begin{equation}\label{xeq9}
D^*(G)\le D(G)\le D^*(G)+5(n_1-1),
\end{equation}
where $D^*(G)=n_1+n_2+n_3-2.$
\end{hip}
We conclude with an application of Theorem \ref{th14a}. If $F$ is a~set of prime integers, then we shall refer to a positive integer each of whose prime factors belong to $F$ as a smooth over a set $F$. The  smooth numbers are related to the Quadratic sieve, and are imported in cryptography in the fastest known integer factorization algorithms. Let $|F|=r.$
We denote by $c(n,r)$ the least positive integer $t$ such that any sequence $S,$ of length $t,$ of smooth integers over $F,$ has a nonempty subsequence $S'$ such that the product of all the terms pf $S'$ is an $n$th power of integer. It is known that $c(n,r)=D(C_n^r)$ see \cite[Theorem 1.6]{ChMG}. Thus by Corollary \ref{nowyap} we obtain:
\begin{thm}
If $n\ge 2$ integer, then
\begin{equation}
c(n,3)\le \min\{20369,3^{\omega(n)}\}(n-1)+1.
\end{equation}
\end{thm}
\normalsize \baselineskip=17pt


Maciej Zakarczemny\\
Institute of Mathematics\\
Cracow University of Technology\\
Warszawska 24\\
31-155 Krak\'ow, Poland\\
E-mail: mzakarczemny@pk.edu.pl

\begin{thebibliography}{99}
\bibitem[1]{AGP} W.R. Alford, A. Granville and C. Pomerance \emph{There are infinitely many Carmichael
numbers, Annals of Math.} 140 (3) (1994), 703-722.
\bibitem[2]{AlD} N. Alon and M. Dubiner, \emph{A lattice point problem and additive number theory}, Combinatorica 15
(1995), 301-309.
\bibitem[3]{ChMG} M. N. Chintamani, B. K. Moriya, W. D. Gao, P. Paul, and R. Thangadurai, \emph{New upper bounds for the Davenport
and for the~Erd\"os-Ginzburg-Ziv constants.} Archiv der Mathematik, 98(2), (2012). 133-142.
\bibitem[4]{EEGKR} Y. Edel, CH. Elsholtz, A. Geroldinger, S. Kubertin, L. Rackham, \emph{Zero-sum problems in finite abelian groups and affine caps}, Quart. J. Math. 58 (2007), 159-186.
\bibitem[6]{EBKAIII} P. van Emde Boas and D. Kruyswijk, \emph{A combinatorial problem on finite abelian
groups. III}, Math. Centrum Amsterdam Afd. Zuivere Wisk 1969 ZW-008.
\bibitem[7]{YGQZ} Yushuang Fan, W. Gao, and Qinghai Zhong, \emph{On the Erd\"os-Ginzburg-Ziv constant of finite abelian groups of
high rank}, J. Number Theory 131 (2011), 1864-1874.
\bibitem[8]{FS10} M. Freeze, W.A. Schmid, \emph{Remarks on a generalization of the Davenport constant}, Discrete Math. 310 (2010),
3373 - 3389.
\bibitem[9]{GGZ} W. Gao and A. Geroldinger, \emph{Zero-sum problems in finite abelian groups: a survey}, Expo. Math.
24 (2006), 337-369.
\bibitem[10]{GT} W. Gao and R. Thangadurai, \emph{On zero-sum sequences of prescribed length}, Aequationes Math. 72
(2006), 201-212.
\bibitem[11]{GHK} A. Geroldinger, F. Halter-Koch, \emph{Non-Unique Facorizations. Algebraic, Combinatorial and
Analytic Theory}, Pure and Applied Mathematics 278, Chapman \& Hall/CRC, Boca Raton, 2006.
\bibitem[12]{GSch} A. Geroldinger, R. Schneider,\emph{ On Davenport's constant. Journal of Combinatorial Theory}, Series A, 61(1), (1992), 147-152.
\bibitem[13]{FH} F. Halter-Koch \emph{A generalization of Davenport's constant and its arithmetical applications}, Colloq. Math. 63, No. 2 (1992), 203-210.
\bibitem[14]{RG} I. Ruzsa and A. Geroldinger, \emph{Combinatorial Number Theory and Additive Group Theory} Advanced Courses in Mathematics CRM Barcelona, Birkh\"auser, Basel, 2009.
\bibitem[15]{BGAA} B. Girard, \emph{An asymptotically tight bound for the Davenport constant}, J. Ec. polytech. Math. 5 (2018), 605-611.
\bibitem[16]{BGWS1} B. Girard and W.A. Schmid, \emph{Direct zero-sum problems for certain groups of rank three}, J. Number Theory,\\https://www.sciencedirect.com/\\science/article/pii/S0022314X18302518.
\bibitem[17]{BGWS2} B. Girard and W.A. Schmid, \emph{Inverse zero-sum problems for certain groups of rank three}, https://arxiv.org/abs/1809.03178.
\end{thebibliography}
\end{document}